\documentclass{amsart}

\usepackage{amssymb, amsmath}
\usepackage{graphics}
\usepackage{pst-all}
\usepackage{color}
\newtheorem{theorem}{Theorem}[section]
\newtheorem{lemma}[theorem]{Lemma}
\newtheorem{e-definition}[theorem]{Definition\rm}
\newtheorem{question}[theorem]{Question}
\newcommand\comp{\mathbin{\scriptscriptstyle\circ}}
\newcommand\Div{\mathrm{Div}}
\newcommand\dive{\mathbin{\preccurlyeq}}
\newcommand\Exp[2]{\vert#1\vert_{#2}}
\newcommand\HS[1]{\hspace{#1ex}}
\newcommand\opR{\HS{0.2}{\star}\HS{0.2}}
\newcommand\pdots{\HS{0.2}{\cdot}{\cdot}{\cdot}\HS{0.2}}
\newcommand\PRES[2]{\langle#1\,\vert\, #2\rangle}
\newcommand\resp{\mbox{\it resp}.\ }
\newcommand\tta{\mathtt{a}}
\newcommand\ttb{\mathtt{b}}
\newcommand\ttc{\mathtt{c}}
\newcommand\wdots{, ...\HS{0.2},}

\begin{document}

\title[Coxeter-like groups for groups of Yang--Baxter equation]{Coxeter-like groups for groups of set-theoretic solutions of the Yang--Baxter equation}
\author{Patrick Dehornoy}
\email{patrick.dehornoy@unicaen.fr}
\address{Laboratoire de Math\'ematiques Nicolas Oresme, CNRS UMR6139, Universit\'e de Caen, 14032 Caen cedex}

\maketitle

\begin{abstract}
We attach with every finite, involutive, nondegenerate set-theoretic solution of the Yang--Baxter equation a finite group that plays for the associated structure group the role that a finite Coxeter group plays for the associated Artin--Tits group. 
\end{abstract}

\section{Introduction}

A \emph{set-theoretic solution of the Yang--Baxter equation} (YBE) is a pair~$(X, R)$ where $R$ is a bijection from~$X^2$ to itself satisfying~$R^{12} R^{23} R^{12} = R^{23} R^{12} R^{23}$, in which $R^{ij} : X^3 \to X^3$ corresponds to $R$ acting in positions~$i$ and~$j$. Set-theoretic solutions of YBE provide particular solutions of the (quantum) Yang--Baxter equation, and received some attention in recent years.

A set-theoretic solution $(X,R)$ of YBE is called \emph{involutive} for$R^2 = \mathrm{id}$, and \emph{nondegenerate} if, writing $R(x, y) = (R_1(x, y), R_2(x, y))$, the maps $y \mapsto R_1(x, y)$ and $x \mapsto R_2(x, y)$ are one-to-one. In this case, the group (\resp monoid) presented by $\PRES{X}{\{xy = z, t \mid x, y, z, t \in X \mbox{ and }R(x, y) = (z, t)\}}$ is called the \emph{structure group} (\resp \emph{structure monoid}) of~$(X, R)$~\cite{Eti}.

Such structure groups happen to admit a number of alternative definitions and make an interesting family. Among others, every structure group is a \emph{Garside group}~\cite{Cho}, meaning that there exists a pair~$(M, \Delta)$ such that $M$ is a cancellative monoid in which left-divisibility---defined by $g \dive h \Leftrightarrow \exists h' {\in} M (h = gh')$---is a lattice, $\Delta$ is a \emph{Garside element} in~$M$---meaning that the left- and right-divisors coincide, are finite in number, and generate~$M$---and $G$ is a group of fractions for~$M$~\cite{Dgk}.

In the case of Artin's braid group~$B_n$, the seminal example of a Garside group, the Garside structure $(B_n^+, \Delta_n)$ is connected with the exact sequence $1 \to P_n \to B_n \to \mathfrak{S}_n \to 1$, where $P_n$ is the pure braid group: $B_n^+$ is the monoid of positive braids, the lattice made by the divisors of~$\Delta_n$ in~$B_n^+$ is isomorphic to the weak order on~$\mathfrak{S}_n$. A presentation of~$\mathfrak{S}_n$ is obtained by adding $n-1$~relations~$\sigma_i^2 = 1$ to the standard presentation of~$B_n$, and the germ derived from~$\mathfrak{S}_n$ and the transpositions~$\sigma_i$, meaning the substructure of~$\mathfrak{S}_n$ where multiplication is restricted to the cases when lengths add, generates~$B_n^+$~\cite{Dif} and its Cayley graph is the Hasse diagram of the divisors of~$\Delta_n$. The results extend to all types in the Cartan classification, connecting spherical Artin--Tits groups with the associated finite Coxeter group.

As Garside groups extend spherical type Artin--Tits groups in many respects, it is natural to ask:

\begin{question}
\label{Ques}
Assume that $G$ is a Garside group, with Garside structure $(M, \Delta)$. Does there exist a finite quotient~$W$ of~$G$ such that $W$ provides a Garside germ for~$M$ and the Cayley graph of~$W$ with respect to the atoms of~$M$ is isomorphic to the lattice of divisors of~$\Delta$ in~$M$ ?
\end{question}

In other words, does every Garside group admit some Coxeter-like group? The general question remains open. The aim of this note is to establish a positive answer for structure groups of set-theoretic solutions of YBE. We attach with every such solution a number called its \emph{class} and establish:

\begin{theorem}
\label{Main}
Assume that $G$ (\resp $M$) is the structure group (\resp monoid) of an involutive, nondegenerate solution~$(X, R)$ of YBE with $X$ of size~$n$ and class~$p$. Then there exist a Garside element~$\Delta$ in~$M$ and a finite group~$W$ of order~$p^n$ entering a short exact sequence $1 \to \mathbb{Z}^n \to G \to W \to 1$ such that $(W, X)$ provides a germ for~$M$ whose Cayley graph is the Hasse diagram of the divisors of~$\Delta$ in~$M$. A presentation of~$W$ is obtained by adding $n$~relations $w_x = 1$ to that of~$G$, with $w_x$ an explicit length~$p$ word beginning with~$x$.
\end{theorem}

Theorem~\ref{Main} extends the results of~\cite{ChG}, in which solutions of class~$2$ are addressed by a different method. Our approach relies on the connection with the \emph{right-cyclic law} of~\cite{Rum} and on the existence of an \emph{$I$-structure}~\cite{GaV}~\cite{JeO} which enables one to carry to arbitrary structure monoids results that are trivial in the case of~$\mathbb{Z}^n$.

\section{The class of a finite RC-quasigroup}

The first step consists in switching from solutions of YBE to RC-quasigroups. 

\begin{e-definition}
\label{RC}
\rm An \emph{RC-system} is a pair~$(X, \opR)$ with $\opR$ a binary operation on~$X$ that obeys the \emph{RC-law} $(x \opR y) \opR (x \opR z) = (y \opR x) \opR (y \opR z)$. An \emph{RC-quasigroup} is an RC-system in which the maps $y \mapsto x \opR y$ are bijections. An RC-quasigroup is \emph{bijective} if the map $(x, y) \mapsto (x \opR y, y \opR x)$ from~$X^2$ to~$X^2$ is bijective. The \emph{associated group} (\resp \emph{monoid}) is presented by $\PRES{X}{\{x(x \opR y) = y(y \opR x) \mid x, y \in\nobreak X\}}$.
\end{e-definition}

As proved in~\cite{Rum}, if $(X, R)$ is an involutive, nondegenerate set-theoretic solution of YBE, then defining~$x \opR y$ to be the (unique)~$z$ satisfying $R_1(x, z) = y$ makes~$X$ into a bijective RC-quasigroup and the group and monoid associated with~$(X, \opR)$ coincide with those of~$(X, R)$. Conversely, every bijective RC-quasigroup~$(X, \opR)$ comes associated with a set-theoretic solution of~YBE. Thus investigating structure groups of set-theoretic solutions of YBE and groups of bijective RC-quasigroups are equivalent tasks. 

\begin{e-definition}
\label{Class}
\rm Inductively define $\Pi_1(x_1) = x_1$ and 
\begin{equation}
\Pi_n(x_1 \wdots x_n) = \Pi_{n-1}(x_1 \wdots x_{n-1}) \opR \Pi_{n-1}(x_1 \wdots \break x_{n-2}, x_n).
\end{equation}
An RC-quasigroup $(X, \opR)$ is said to be of \emph{class~$p$} if $\Pi_{p+1}(x \wdots x, y) = y$ holds for all~$x, y$ in~$X$.
\end{e-definition}

\begin{lemma}
\label{Finite}
Every finite RC-quasigroup is of class~$p$ for some~$p$.
\end{lemma}

\begin{proof}
Let $(X, \opR)$ be a finite RC-quasigroup. First, $(X, \opR)$ must be bijective, that is, the map $\Psi : (x, y) \mapsto (x \opR y, y \opR x)$ is bijective on~$X^2$~\cite{Rum} (or \cite{JeO} for a different argument). Now, consider the map $\Phi : (x, y) \mapsto (x \opR x, x \opR y)$ on~$X^2$. Assume $(x, y) \not= (x', y')$. For $x \not= x'$, $\Psi(x, x) \not= \Psi(x', x')$ implies $x \opR x \not= x' \opR x'$; for $x =\nobreak x'$, we have $y \not= y'$, whence $x \opR y \not= x \opR y'$ since left-translations are injective; so $\Phi(x, y) \not= \Phi(x', y')$ always holds. So $\Phi$ is injective, hence bijective on~$X^2$, and $\Phi^{p+1} = \mathrm{id}$ holds for some~$p \ge 1$. An induction gives $\Phi^r(x, y) = (\Pi_r(x \wdots x, x), \Pi_r(x \wdots x, y))$. So $\Phi^{p+1} = \mathrm{id}$ implies $\Pi_{p+1}(x \wdots x, y) = y$ for all~$x, y$, that is, $(X, \opR)$ is of class~$p$.  
\end{proof}

\section{Using the $I$-structure}

From now on, assume that $M$ (\resp $G$) is the structure group of some finite RC-quasigroup~$(X, \opR)$ of size~$n$ and class~$p$. The form of the defining relations of~$M$ implies that the Cayley graph of~$M$ with respect to~$X$ is an $n$-dimensional lattice. It was proved in~\cite{GaV} that $M$ admits a \emph{(right) $I$-structure}, defined to be a bijection $\nu: \mathbb{N}^n \to M$ satisfying $\nu(1) = 1$ and $\{\nu(u x) \mid x \in X\} = \{\nu(u)x \mid x \in X\}$ for every~$u$ in~$\mathbb{N}^n$, that is, equivalently, $\nu(u x) = \nu(u)\pi(u)(x)$ for some permutation~$\pi(u)$ of~$X$. The monoid~$M$ is then called \emph{of right-$I$-type}. Our point is that the $I$-structure (which is unique) is connected with~$\opR$. Without loss of generality, we shall assume that $X$ is the standard basis of~$\mathbb{N}^n$ and that $\nu(x) = x$ holds for~$x$ in~$X$.

\begin{lemma}
\label{Nu}
For all $x_1 \wdots x_r$ in~$X$, we have $\nu(x_1 \pdots x_r) = \Sigma_r(x_1 \wdots x_r)$, with $\Sigma_r$ inductively defined by 
$\Sigma_1(x_1) = x_1$ and $$\Sigma_r(x_1 \wdots x_r) = \Sigma_{r-1}(x_1, \wdots x_{r-1}) \cdot \Pi_r(x_1 \wdots x_r).$$
\end{lemma}

\begin{proof}
The result can be established directly by developing a convenient RC-calculus and proving that $\Sigma_r(x_1 \wdots x_r)$ satisfies all properties required for an $I$-structure. A shorter proof is to start from the existence of the $I$-structure~$\nu$ and just connect it with the values of~$\Sigma_r$. As established in~\cite{GaV} (see also \cite[Chapter~8, Lemma~8.2.2]{JeO}), the following inductive relations are satisfied for all~$u, v$ in~$\mathbb{N}^n$: 
\begin{equation}
\label{Product}
\nu(uv) = \nu(u)\, \nu(\pi(u)[v]) \mbox{\quad and\quad } \pi(uv) = \pi(\pi(u)[v]) \comp \pi(u)
\end{equation}
where $\pi[u]$ is the result of applying~$\pi$ to~$u$ componentwise.

We then use induction on~$r$. For $r = 1$, the result is obvious. Assume $r = 2$ and $x_1 \not= x_2$. By definition, we have $\nu(x_1 x_2) = x_1 \pi(x_1)(x_2) = \nu(x_2 x_1) = x_2 \pi(x_2)(x_1)$. This shows that $\nu(x_1 x_2)$ must be the right-lcm (least common right-multiple) of~$x_1$ and~$x_2$ in~$M$. On the other hand, $x_1 (x_1 \opR x_2) = x_2 (x_2 \opR x_1)$ holds in~$M$ by definition, and this must also represent the right-lcm of~$x_1$ and~$x_2$. By uniqueness of the right-lcm and left-cancellativity, we deduce $\pi(x_1)(x_2) = x_1 \opR x_2$. Next, for $x_1 = x_2$, the value of~$\pi(x_1)(x_2)$, as well as that of~$x_1 \opR x_2$, must be the unique element of~$X$ that is not of the form $\pi(x_1)(x)$ or $x_1 \opR x$ with $x \not= x_1$, respectively. This forces $\pi(x_1)(x_2) = x_1 \opR x_2$ in this case as well, implying $\nu(x_1 x_2) = x_1(x_1 \opR x_2) = \Sigma_2(x_1, x_2)$ in every case. Assume now $r \ge 3$. We find
\begin{multline*}
\nu(x_1 \pdots x_r) 
= x_1 \, \nu(\pi(x_1)[x_2 \pdots x_r])
= x_1 \, \nu((x_1 \opR x_2) \pdots (x_1 \opR x_r))\\
= x_1 \, \Sigma_{r-1}(x_1 \opR x_2 \wdots x_1 \opR x_r)
= \Sigma_r(x_1, x_2 \wdots x_r),
\end{multline*}
the first equality by~(\ref{Product}), the second by the case~$r = 2$, the third by the induction hypothesis, and the last one by expanding the terms.
\end{proof}

\begin{lemma}
\label{Frozen}
For~$x \in X$ and $r \ge 0$, let $x^{[r]} = \nu(x^r)$. For all $x \in X$ and $u \in \mathbb{N}^n$, we have $\nu(x^p u) = x^{[p]} \nu(u)$ in~$M$. In particular, we have $\pi(x^p) = \mathrm{id}$ and, for all~$x, y$ in~$X$, the elements $x^{[p]}$ and $y^{[p]}$ commute in~$M$.
\end{lemma}

\begin{proof}
Let $y_1 \pdots y_q$ be a decomposition of~$u$ in terms of elements of~$X$. By Lemma~\ref{Nu}, we have 
\begin{multline*}
\nu(x^pu) = \Sigma_{p+q}(x \wdots x, y_1 \wdots y_q)\\
 = \Sigma_p(x \wdots x) \Sigma_q(\Pi_{p+1}(x \wdots x, y_1) \wdots \Pi_{p+1}(x, \wdots x, y_q))\\
 = \Sigma_p(x \wdots x) \Sigma_q(y_1 \wdots y_q) = \nu(x^p) \nu(y_1 \pdots y_q) = x^{[p]} \nu(u),
 \end{multline*}
in which the second equality comes from expanding the terms and the third one from the assumption that $M$ is of class~$p$. Applying with $u = y$ in~$X$ and merging with $\nu(x^p y) = \nu(x^p) \, \pi(x^p)(y)$, we deduce $\pi(x^p) = \mathrm{id}$. On the other hand, applying with $u = y^{[p]}$, we find $x^{[p]}y^{[p]} = \nu(x^py^p) = \nu(y^px^p) = y^{[p]}x^{[p]}$.
\end{proof}

\begin{lemma}
\label{Delta}
Assume $p \ge 2$ and define $\Delta = \nu(\prod_{x \in X}x^{p-1})$. Then $\Delta$ is a Garside element in~$M$, and its family of divisors is $\nu(\{0 \wdots p-1\}^n)$, which has $p^n$~elements. Moreover $\Delta^p$ is central in~$M$.
\end{lemma}

\begin{proof}
The map~$\nu$ is compatible with~$\dive$ : for all~$u, v$ in~$\mathbb{N}^n$, we have $u \dive v$ in~$\mathbb{N}^n$ if and only if $\nu(u) \dive \nu(v)$ holds in~$M$. Indeed, by~(\ref{Product}), $v = u x$ with~$x$ in~$X$ implies $\nu(v) = \nu(u) \pi(u)(x)$, whence $\nu(u) \dive \nu(v)$ in~$M$. Conversely, for $\nu(v) = \nu(u)x$ with~$x$ in~$X$, as $\pi(u)$ is bijective, we have $\pi(u)(y) = x$ for some~$y$ in~$X$, whence $\nu(u y) = \nu(u)\pi(u)(y) = \nu(u)x = \nu(v)$, and $v = u y$ since $\nu$ is injective, that is, $u \dive v$ in~$\mathbb{N}^n$. Hence the left-divisors of~$\Delta$ in~$M$ are the image under~$\nu$ of the $p^n$~divisors of~$\delta^{p-1}$ in~$\mathbb{N}^n$, with $\delta = \prod_{x \in X}x$. For right-divisors, the maps~$\pi(u)$ are bijective, so every right-divisor of~$\Delta$ must be a left-divisor of~$\Delta$. Then the duality map $g \mapsto h$ for $gh = \Delta$ is a bijection from the left- to the right-divisors of~$\Delta$. So the left- and right-divisors of~$\Delta$ coincide, and they are $p^n$ in number. Since every element of~$X$ divides~$\Delta$, the latter is a Garside element in~$M$. Finally, by Lemma~\ref{Nu}, $\Delta^p$ is the product of the elements~$x^{[p]}$ repeated $p-1$~times; as $\sigma[\delta] = \delta$ holds for every permutation~$\sigma$, we deduce $x\Delta^p = \Delta^px$ for every~$x$. 
\end{proof}

For~$u \in \mathbb{N}^n$ and $x \in X$, write $\Exp{u}{x}$ for the (well-defined) number of~$x$ in an $X$-decomposition of~$u$.

\begin{lemma}
\label{Congruence}
For~$u, u'$ in~$\mathbb{N}^n$, say that $u \equiv_p u'$ holds if, for every~$x$ in~$X$, we have $\Exp{u}{x} = \Exp{u'}{x} \bmod{p}$, and, for $g, g'$ in~$M$, say that $g \equiv g'$ holds for $\nu^{-1}(g) \equiv_p \nu^{-1}(g')$. Then $\equiv$ is an equivalence relation on~$M$ that is compatible with left- and right-multiplication. 
\end{lemma}

\begin{proof}
As $\nu$ is bijective, carrying the equivalence relation~$\equiv_p$ of~$\mathbb{N}^n$ to~$M$ yields an equivalence relation. Assume $\nu(u) \equiv \nu(u')$. Without loss of generality, we may assume $u' = u x^p = x^p u$ with $x$ in~$X$. Applying~(\ref{Product}) and Lemma~\ref{Frozen}, we deduce $\pi(u) = \pi(u')$ and, therefore, $\nu(u) \pi(u)(y) = \nu(u y) \equiv \nu(u' y) = \nu(u') \pi(u)(y)$. As $\pi(u)(y)$ takes every value in~$X$ when $y$ varies, $\equiv$ is compatible with right-multiplication by~$X$. On the other hand, $u \equiv_p u'$ implies $\sigma[u] \equiv_p \sigma[u']$ for every permutation~$\sigma$ in~$\mathfrak{S}_X$, so we obtain $y \nu(u) = \nu(y \pi(y)^{-1}[u]) \equiv \nu(y \pi(y)^{-1}[u']) = y \nu(u')$, and $\equiv$ is compatible with left-multiplication by~$X$. 
\end{proof}

\begin{lemma}
\label{Extension}
For $g = \Delta^{pe} h$, $g' = \Delta^{pe'} h'$  in~$G$ with $h, h' \in M$, say that $g \equiv g'$ holds if $h \equiv h'$ does. Then $\equiv$ is a congruence on~$G$ with $p^n$~classes, and the kernel of $G \to G{/}{\equiv}$ is the Abelian subgroup of~$G$ generated by the elements~$x^{[p]}$ with~$x \in X$.
\end{lemma}

\begin{proof}
As $\Delta$ is a Garside element in~$M$, every element of~$G$ admits a (non-unique) expression $\Delta^{pe} h$ with $e \in \mathbb{Z}$ and $h \in M$. Assume $g = \Delta^{pe} h = \Delta^{pe_1} h_1$ with $e > e_1$. As $M$ is left-cancellative, we find $h_1 = \Delta^{p(e - e_1)} h$, whence $h_1 \equiv h$. So, for every~$h'$ in~$M$, we have $h \equiv h' \Leftrightarrow h_1 \equiv h'$ and $\equiv$ is well-defined on~$G$. That $\equiv$ is compatible with multiplication on~$G$ follows from the compatibility on~$M$ and the fact that $\Delta^p$ lies in the centre of~$G$. Next, by definition, every element of~$G$ is $\equiv$-equivalent to some element of~$M$, so the number of $\equiv$-classes in~$G$ equals the number of $\equiv$-classes in~$M$, hence the number~$p^n$ of $\equiv_p$-classes in~$\mathbb{N}^n$. 

Finally, $u \equiv_p x^p u$ holds for all~$x$ in~$X$ and $u$ in~$\mathbb{N}^n$. This, together with Lemma~\ref{Nu}, implies $x^{[p]} \equiv 1$. Conversely, assume $g \equiv 1$ in~$M$. By definition, $\nu^{-1}(g)$ lies in the $\equiv_p$-class of~$1$, hence one can go from~$\nu^{-1}(g)$ to~$1$ by multiplying or dividing by elements~$x^p$ with~$x \in X$. By Lemma~\ref{Nu} again, this means that one can go from~$g$ to~$1$ by multiplying or dividing by elements~$x^{[p]}$ with~$x \in X$. In other words, the latter elements generate the kernel of the projection of~$G$ to~$G{/}{\equiv}$.
\end{proof}

Now Theorem~\ref{Main} readily follows. Indeed, define~$W$ to be the finite quotient-group~$G{/}{\equiv}$. We saw that the kernel of the projection of~$G$ onto~$W$ is the free Abelian group generated by the $n$~elements~$x^{[p]}$ with~$x \in X$, thus giving an exact sequence $1 \to \mathbb{Z}^n \to G \to W \to 1$. A presentation of~$W$ is obtained by adding to the presentation of~$G$ in Definition~\ref{RC} the $n$~relations $x^{[p]} = 1$, that is, $x(x\opR x)((x \opR x)\opR(x \opR x)) ... = 1$. By construction, the Hasse diagram of the lattice made of the $p^n$~divisors of~$\Delta$ is the image under~$\nu$ of the sublattice of~$\mathbb{N}^n$ made of the $p^n$ divisors of~$\delta$ in~$\mathbb{N}^n$, whereas the Cayley graph of the germ derived from~$(W, X)$---that is, $W$ equipped with the partial product obtained by restricting to the cases when the $X$-lengths add---is the image under~$\nu$ of the Cayley graph of the germ derived from the quotient-group~$\mathbb{Z}^n{/}{\equiv_p}$: the (obvious) equality in the case of~$\mathbb{N}^n$ implies the equality in the case of~$M$. 

\section{An example}

For an RC-quasigroup of class~1, that is, satisfying $x \opR y = y$ for all~$x, y$, the group~$G$ is a free Abelian group, the group~$W$ is trivial, and the short exact sequence of Theorem~\ref{Main} reduces to $1 \to \mathbb{Z}^n \to G \to 1$.

Class~2, that is, when $(x \opR x) \opR (x \opR y) = y$ holds for all~$x, y$, is addressed in~\cite{ChG} (with no connection with RC-quasigroups). The element~$\Delta$ is the right-lcm of~$X$, it has $2^n$~divisors which are the right-lcms of subsets of~$X$, and the group~$W$ is the order~$2^n$ quotient of~$G$ obtained by adding the relations $x(x \opR x) = 1$. 

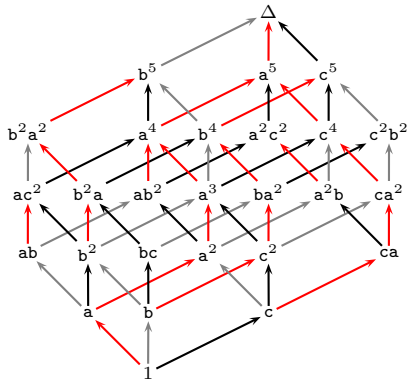
\begin{figure}[b]
\begin{picture}(48,48)(0,2)
\psset{unit=1mm}
\psset{linewidth=0.8pt}
\psset{arrowsize=3pt}
\psset{dotsep=1.5pt}
\psset{dash=2pt 1pt}
\psset{doublesep=0.5pt}
\newpsstyle{a}{linecolor=red}
\newpsstyle{b}{linecolor=gray}
\newpsstyle{c}{linecolor=black}
\rput(16,0){\rnode[c]{1}{$\scriptstyle1$}}
\rput(8,8){\rnode[c]{a}{$\scriptstyle\tta$}}
\rput(16,8){\rnode[c]{b}{$\scriptstyle\ttb$}}
\rput(32,8){\rnode[c]{c}{$\scriptstyle\ttc$}}
\rput(0,16){\rnode[c]{ab}{$\scriptstyle\tta\ttb$}}
\rput(8,16){\rnode[c]{bb}{$\scriptstyle\ttb^2$}}
\rput(16,16){\rnode[c]{bc}{$\scriptstyle\ttb\ttc$}}
\rput(24,16){\rnode[c]{aa}{$\scriptstyle\tta^2$}}
\rput(32,16){\rnode[c]{cc}{$\scriptstyle\ttc^2$}}
\rput(48,16){\rnode[c]{ca}{$\scriptstyle\ttc\tta$}}
\rput(0,24){\rnode[c]{acc}{$\scriptstyle\tta\ttc^2$}}
\rput(8,24){\rnode[c]{bba}{$\scriptstyle\ttb^2\tta$}}
\rput(16,24){\rnode[c]{abb}{$\scriptstyle\tta\ttb^2$}}
\rput(24,24){\rnode[c]{aaa}{$\scriptstyle\tta^3$}}
\rput(32,24){\rnode[c]{baa}{$\scriptstyle\ttb\tta^2$}}
\rput(40,24){\rnode[c]{aab}{$\scriptstyle\tta^2\ttb$}}
\rput(48,24){\rnode[c]{caa}{$\scriptstyle\ttc\tta^2$}}
\rput(0,32){\rnode[c]{bbaa}{$\scriptstyle\ttb^2\tta^2$}}
\rput(16,32){\rnode[c]{aaaa}{$\scriptstyle\tta^4$}}
\rput(24,32){\rnode[c]{bbbb}{$\scriptstyle\ttb^4$}}
\rput(32,32){\rnode[c]{aacc}{$\scriptstyle\tta^2\ttc^2$}}
\rput(40,32){\rnode[c]{cccc}{$\scriptstyle\ttc^4$}}
\rput(48,32){\rnode[c]{ccbb}{$\scriptstyle\ttc^2\ttb^2$}}
\rput(16,40){\rnode[c]{bbbbb}{$\scriptstyle\ttb^5$}}
\rput(32,40){\rnode[c]{aaaaa}{$\scriptstyle\tta^5$}}
\rput(40,40){\rnode[c]{ccccc}{$\scriptstyle\ttc^5$}}
\rput(32,48){\rnode[c]{Delta}{$\scriptstyle\Delta$}}
\psset{nodesep=0.4mm}
\ncline[style=a]{->}1a
\ncline[style=b]{->}1b
\ncline[style=c]{->}1c
\ncline[style=b]{->}a{ab}
\ncline[style=c]{->}a{bb}
\ncline[style=a]{->}a{aa}
\ncline[style=b]{->}b{bb}
\ncline[style=c]{->}b{bc}
\ncline[style=a]{->}b{cc}
\ncline[style=b]{->}c{aa}
\ncline[style=c]{->}c{cc}
\ncline[style=a]{->}c{ca}
\ncline[style=a]{->}{ab}{acc}
\ncline[style=b]{->}{ab}{abb}
\ncline[style=a]{->}{bb}{bba}
\ncline[style=b]{->}{bb}{aaa}
\ncline[style=c]{->}{bb}{acc}
\ncline[style=b]{->}{bc}{baa}
\ncline[style=c]{->}{bc}{bba}
\ncline[style=a]{->}{aa}{aaa}
\ncline[style=b]{->}{aa}{aab}
\ncline[style=c]{->}{aa}{abb}
\ncline[style=a]{->}{cc}{baa}
\ncline[style=b]{->}{cc}{caa}
\ncline[style=c]{->}{cc}{aaa}
\ncline[style=a]{->}{ca}{caa}
\ncline[style=c]{->}{ca}{aab}
\ncline[style=b]{->}{acc}{bbaa}
\ncline[style=c]{->}{acc}{aaaa}
\ncline[style=a]{->}{bba}{bbaa}
\ncline[style=c]{->}{bba}{bbbb}
\ncline[style=a]{->}{abb}{aaaa}
\ncline[style=c]{->}{abb}{aacc}
\ncline[style=a]{->}{aaa}{aaaa}
\ncline[style=b]{->}{aaa}{bbbb}
\ncline[style=c]{->}{aaa}{cccc}
\ncline[style=a]{->}{baa}{bbbb}
\ncline[style=c]{->}{baa}{ccbb}
\ncline[style=a]{->}{aab}{aacc}
\ncline[style=b]{->}{aab}{cccc}
\ncline[style=a]{->}{caa}{cccc}
\ncline[style=b]{->}{caa}{ccbb}
\ncline[style=a]{->}{bbaa}{bbbbb}
\ncline[style=a]{->}{aaaa}{aaaaa}
\ncline[style=c]{->}{aaaa}{bbbbb}
\ncline[style=a]{->}{bbbb}{ccccc}
\ncline[style=b]{->}{bbbb}{bbbbb}
\ncline[style=c]{->}{aacc}{aaaaa}
\ncline[style=a]{->}{cccc}{aaaaa}
\ncline[style=c]{->}{cccc}{ccccc}
\ncline[style=b]{->}{ccbb}{ccccc}
\ncline[style=b]{->}{bbbbb}{Delta}
\ncline[style=a]{->}{aaaaa}{Delta}
\ncline[style=c]{->}{ccccc}{Delta}
\end{picture}
\caption{An example in class 3: here $W$ has $3^3 = 27$ elements, and its Cayley graph is a cube with edges of length $3-1$.}
\label{Cube}
\end{figure}

For one example in class~3, consider $\{\tta, \ttb, \ttc\}$ with $x \opR y = f(y)$, $f : \tta \mapsto \ttb \mapsto \ttc \mapsto \tta$. The associated presentation is $\PRES{\tta, \ttb, \ttc}{\tta\ttc =\ttb^2, \ttb\tta = \ttc^2, \ttc\ttb = \tta^2}$. The  smallest Garside element is~$\tta^3$, but, here, in class~$3$, we consider the next one, namely~$\Delta =\nobreak \tta^6$. Adding to the above presentation of~$G$ the three relations $x(x\opR x)((x \opR x)\opR(x \opR x)) = 1$, namely $\tta\ttb\ttc = \ttb\ttc\tta = \ttc\tta\ttb = 1$, here reducing to $\tta\ttb\ttc = 1$, one obtains for~$W$ the presentation $\PRES{\tta, \ttb, \ttc}{\tta\ttc =\nobreak \ttb^2, \ttb\tta = \ttc^2, \ttc\ttb = \tta^2, \tta\ttb\ttc = 1}$. The lattice $\Div(\Delta)$ has 27~elements, its diagram is the cube shown on the right. The latter is also the Cayley graph of the germ derived from~$(W, X)$.

\end{document}